\documentclass[12pt]{amsart}
\usepackage{amssymb,amsmath,amsfonts,latexsym}
\usepackage{bm}
\usepackage[all,cmtip]{xy}
\usepackage{amscd}

\newcommand{\bea}{\begin{eqnarray}}
\newcommand{\eea}{\end{eqnarray}}

\def\textmatrix#1&#2\\#3&#4\\{\bigl({#1 \atop #3}\ {#2 \atop #4}\bigr)}
\def\dispmatrix#1&#2\\#3&#4\\{\left({#1 \atop #3}\ {#2 \atop #4}\right)}
\newcommand{\be}{\begin{equation}}
\newcommand{\ee}{\end{equation}}
\newcommand{\ben}{\begin{eqnarray*}}
\newcommand{\een}{\end{eqnarray*}}

\newcommand{\NI}{\noindent}

\newcommand{\bi}{\begin{itemize}}
\newcommand{\ei}{\end{itemize}}

\newtheorem{Theorem}{\sc Theorem}[section]
\newtheorem{Lemma}[Theorem]{\sc Lemma}
\newtheorem{Proposition}[Theorem]{\sc Proposition}

\theoremstyle{definition}

\newtheorem{Remark}[Theorem]{\sc Remark}

\theoremstyle{plain}

\newtheorem{thm}{Theorem}[section]
\newtheorem{cor}[thm]{Corollary}

\theoremstyle{definition}

\numberwithin{equation}{section}

\let\phi=\varphi

\begin{document}

\title[On the improvements of Hardy and Copson inequalities]
{On the improvements of Hardy and Copson inequalities}
%
\author[Das]{Bikram Das}
\address{Dr. APJ Abdul Kalam Technical University, Lucknow-226021 \&
Indian Institute of Carpet Technology, Chauri Road, Bhadohi- 221401, Uttar Pradesh, India}
\email{dasb23113@gmail.com}
\author[Manna]{Atanu Manna$^*$ }
\address{Indian Institute of Carpet Technology, Bhadohi-221401, Uttar Pradesh, India}
\email{atanu.manna@iict.ac.in, atanuiitkgp86@gmail.com ($^*$Corresponding author)}

\subjclass[2010]{Primary 26D15; Secondary 46A45.}
\keywords{ Discrete Hardy's inequality; Improvement; Copson's inequality; Sequence space.}

\begin{abstract}
In this current work, we revisit the recent improvements of the discrete Hardy's inequality in one dimension and establish an extended improved discrete Hardy's inequality with its optimality. We also study one-dimensional discrete Copson's inequality (E.T. Copson, \emph{Notes on a series of positive terms}, J. London Math. Soc., 2 (1927), 9-12.), and achieve an improvement of the same in a particular case. Further, we study some fundamental structures such as completeness, K\"{o}the-Toeplitz duality, separability, etc. of the sequence spaces which originated from the improved discrete Hardy and Copson inequalities in one dimension.
\end{abstract}
\maketitle

\section{Introduction}{\label{secint}}
The famous discrete Hardy's inequality was developed in the twentieth century during the period $1906-1928$. Apart from the contribution of G. H. Hardy, the other mathematicians such as E. Landau, G. P\'{o}lya, I. Schur, M. Riesz contributed a lot to the development of it. A systematic survey on the prehistory of Hardy's inequality is well-explained by Kufner et al. \cite{AKR}. Let us recall the one-dimensional discrete Hardy's inequality in its crudest form. For $p>1$, and a sequence $\{a_{n}\}_{n=1}^{\infty}$ of complex numbers the classical discrete Hardy's inequality (\cite{GHYLITTLE}, Theorem 326) in one dimension states that
\begin{align}{\label{DHI}}
&\displaystyle\sum_{n=1}^{\infty}\Big|\frac{1}{n}\sum_{k=1}^{n}a_{k}\Big|^{p}<\Big(\frac{p}{p-1}\Big)^{p}\displaystyle\sum_{n=1}^{\infty}|a_{n}|^{p},
\end{align} holds unless $a_n$ is null. Here $p>1$ is a real no., and the constant term $\big(\frac{p}{p-1}\big)^{p}$ associated with the inequality (\ref{DHI}) is sharp. In a letter, Landau (see \cite{ELU}) wrote to G.H. Hardy, and addressed a proof of the inequality (\ref{DHI}) with a sharp constant. The results of Landau were officially published five years later than the letter of Landau to Schur \cite{LANDAU}. During this period Hardy \cite{GHYNOTE} was working with both continuous as well as discrete cases of Hardy's inequality, and later he also commented on Landau's letter. Since both E. Landau and G. H. Hardy have contributed to the development of the inequality (\ref{DHI}), so this inequality (\ref{DHI}) sometimes called as Hardy-Landau inequality \cite{AKR}.\\
The 1925 articles of G. H. Hardy (\cite{GHYNOTE251}, \cite{GHYNOTE252}) contain many interesting results. One of them is the following extension of Hardy's inequality (\ref{DHI}). Suppose that $\{q_n\}$ sequences of real numbers such that $q_{n}>0$, and denote $A_{n}=q_{1}a_{1}+ q_{2}a_{2}+ \ldots +q_{n}a_{n}$ and $Q_{n}=q_1+q_2+\ldots+q_{n}$ for $n\in \mathbb{N}$. If for $p>1$, $\{q_{n}^{\frac{1}{p}}a_{n}\}\in \ell_p$ then
\begin{align}{\label{GDHI}}
\displaystyle\sum_{n=1}^{\infty}q_nQ_{n}^{-p}|A_{n}|^{p}&< \Big(\frac{p}{p-1}\Big)^{p}\displaystyle\sum_{n=1}^{\infty}q_{n}|a_{n}|^{p},
\end{align} unless all $a_n$ is null. Also the constant term $\big(\frac{p}{p-1}\big)^{p}$ is sharp. The inequality (\ref{GDHI}) is studied and extended in various ways. For instance, E. T. Copson \cite{ECN} introduced and studied an extended version of inequality (\ref{GDHI}) as below. Let $1<c\leq p$. Then
\begin{align}{\label{COPI}}
&\displaystyle\sum_{n=1}^{\infty}q_{n}Q^{-c}_{n}|A_{n}|^{p}\leq\Big(\frac{p}{c-1}\Big)^{p}\displaystyle\sum_{n=1}^{\infty}q_{n}Q^{p-c}_{n}|a_{n}|^{p},
\end{align}where the associated constant term is best possible, equality holds good when all $a_n$ are `$0$'. There are many applications of these inequalities (\ref{DHI}), (\ref{GDHI}), and (\ref{COPI}) can be found in several parts of analysis in the form of generalization, extensions or their direct applications to spectral theory, graph theory, differential equations etc.. A careful study on the analysis and geometry of Hardy's inequality can be found in \cite{BEL}. For various studies on the Hardy inequalities, we refer to \cite{MBN}, \cite{GHYLITTLE}, \cite{KAPLAP}, \cite{LEINGEN}, \cite{PLEFEVRE}, \cite{LIUZ}, \cite{MANNANEW}, and references cited therein for the readers. One of the recent trends of research in operator theory is to improve various inequalities for operators on Hilbert space. Many authors are engaging themselves in investigating improvements of the well-known numerical radius and Berezin number inequalities for bounded linear operators on a Hilbert space $\mathcal{H}$ (see \cite{MAJEE} for the latest research). But as far as the discrete Hardy's inequality is concerned, we couldn't find any point-wise improvement 
in previous years. In 2018, we have a surprising result on point-wise improvement of the discrete classical Hardy's inequality (\ref{DHI}), and it is due to Keller et al. \cite{MKR}. To write briefly their results, we rewrite the Hardy inequality (\ref{DHI}) in a different form, and use the following notation.\\
Denote $C_{c}(\mathbb{N}_0)$ as the space of finitely supported functions on $\mathbb{N}_0=\{0, 1, 2, 3, \ldots\}$. Then clearly for all $A=(A_n)\in C_{c}(\mathbb{N}_0)$, where $A_n=a_1+a_2+\ldots+a_{n}$ with the assumption that $A_0=0$, inequality (\ref{DHI}) is equivalent to the following:
\begin{align}{\label{EFDHI}}
&\displaystyle\sum_{n=1}^{\infty}|A_n-A_{n-1}|^{p}\geq\Big(\frac{p-1}{p}\Big)^{p}\displaystyle\sum_{n=1}^{\infty}\frac{|A_n|^{p}}{n^{p}}.
\end{align}
Recently, Keller et al. \cite{MKR} improved the inequality (\ref{EFDHI}) for the case when $p=2$, and later Fischer et al. \cite{FFR} obtained the improvement of (\ref{EFDHI}) for general $p>1$. In fact the authors in \cite{MKR} have proved that
\begin{align}{\label{IMHI}}
\displaystyle\sum_{n=1}^{\infty}| A_n-A_{n-1}|^{2}&\geq\frac{1}{4}\displaystyle\sum_{n=1}^{\infty} \frac{| A_n|^{2}}{n^{2}}+\displaystyle\sum_{k=2}^{\infty}\binom{4k}{2k}\frac{1}{(4k-1)2^{4k-1}}\displaystyle\sum_{n=2}^{\infty}\frac{|A_n|^{2}}{n^{2k}}>\frac{1}{4}\displaystyle
\sum_{n=1}^{\infty} \frac{|A_n|^{2}}{n^{2}},
\end{align}and which is equivalent to the following
\begin{align}{\label{IMHI2}}
\displaystyle\sum_{n=1}^{\infty}| A_n-A_{n-1}|^{2}&\geq\displaystyle\sum_{n=1}^{\infty} w_n|A_n|^{2}>\frac{1}{4}\displaystyle
\sum_{n=1}^{\infty} \frac{|A_n|^{2}}{n^{2}},
\end{align}where the sequence $w_n$ is defined as follows:
\begin{center}
$w_n=2-\sqrt{1-\frac{1}{n}}-\sqrt{1+\frac{1}{n}}>\frac{1}{4n^2}$, $n\in \mathbb{N}$.
\end{center}
The history of defining such kind of sequence is also very surprising and pretty interesting. The readers are referred to the works of Keller et al. (\cite{MKR}, \cite{MKRGRAPH}) for getting a history. In a recent work of Krej\v{c}i\v{r}\'{i}k and \v{S}tampach \cite{DKK}, a short and elementary proof of inequality (\ref{IMHI}) is presented by proving an identity, which finally contributes a remainder term. The optimality of the sequence $w_n$ is also established by these authors. In a more latest result of Krej\v{c}i\v{r}\'{i}k et al. (see Theorem 10, \cite{DKALFS}), the authors established that
\begin{align}{\label{IMHIGN3}}
\displaystyle\sum_{n=1}^{\infty} w_n(g)|A_n|^{2}&\leq \displaystyle\sum_{n=1}^{\infty}| A_n-A_{n-1}|^{2},
\end{align}
which is an extension of inequality (\ref{IMHI2}), and the sequence $w_n(g)$ is defined as below
\begin{center}
$w_n(g)=2-\frac{g_{n-1}}{g_n}-\frac{g_{n+1}}{g_n}$, $n\in \mathbb{N}$,
\end{center}where $g_n$ is a strictly positive sequence of real numbers with the convention that $g_0=0$. In particular, when $g_n=\sqrt{n}$ one can easily have the inequality (\ref{IMHI2}). The criteria for optimality of $w_n(g)$ was also considered by these researchers \cite{DKALFS}. Using the idea of defining $w_n(g)$, an improvement of Rellich inequality was obtained by Gerhat et al. \cite{BGDKFS}, which is an unpublished preprint now. \\
Observe that, if we substitute $p=2$, $\alpha=2-c$ and $q_n=1$ for all $n\in \mathbb{N}$ in (\ref{COPI}), then it transformed into the following:
\begin{align*}
\displaystyle\sum_{n=1}^{\infty}{n}^{\alpha-2}|A_{n}|^{2}&\leq\frac{4}{(\alpha-1)^{2}}\displaystyle\sum_{n=1}^{\infty}{n}^{\alpha}|a_{n}|^{2},
\end{align*}which further equivalent to the following inequality:
\begin{align}{\label{COPIPAR}}
\displaystyle\sum_{n=1}^{\infty}|A_{n}-A_{n-1}|^{2}{n}^{\alpha}&\geq \frac{(\alpha-1)^2}{4}\sum_{n=1}^{\infty}\frac{|A_{n}|^{2}}{n^2}{n}^{\alpha},
\end{align}where it is pre-assumed that $A_0=0$. A very surprising and noticeable point is that Gupta \cite{SGA} has recently investigated the improvement of inequality (\ref{COPIPAR}), which is called one-dimensional discrete Hardy's inequality with power weights. Indeed Gupta \cite{SGA} obtained the following inequality:
\begin{align}{\label{GUPI}}
\displaystyle\sum_{n=1}^{\infty}|A_{n}-A_{n-1}|^{2}n^{\alpha} &\geq \displaystyle\sum_{n=1}^{\infty}w_{n}(\alpha,\beta)|A_{n}|^{2},
\end{align}
where $\alpha, \beta\in \mathbb{R}$, $w_{1}(\alpha,\beta):=1+2^{\alpha}-2^{\alpha+\beta}$, and for $ n\geq2$
\begin{align*}
w_{n}(\alpha,\beta)& =n^{\alpha}\Big[1+\Big(1+\frac{1}{n}\Big)^{\alpha}-\Big(1-\frac{1}{n}\Big)^{\beta}-\Big(1+\frac{1}{n}\Big)^{\alpha+\beta}\Big].
\end{align*}
The importance of inequality (\ref{GUPI}) is of two types, one is it contains Hardy's inequality (\ref{COPIPAR}) with power weights whenever $\alpha\in [0, 1)\cup[5,\infty)$, and another is the improvement of (\ref{COPIPAR}) for the case when $\alpha\in [\frac{1}{3}, 1)\cup \{0\}$.
\begin{Remark}
It is to be noted that in the case when $\alpha=0$, and $\beta=\frac{1}{2}$, then inequality (\ref{GUPI}) transformed into the improved discrete Hardy's inequality (\ref{IMHI}) derived by Keller et al. \cite{MKR}.
\end{Remark}
With the above discussions, one may ask the following
\begin{center}
\emph{(a) Is it possible to extend both the inequalities (\ref{IMHIGN3}), and (\ref{GUPI}) for more general weight sequence, say $\{\lambda_n\}$?}
\end{center}
\vspace{0.2cm}
Now let us choose another substitutions $p=2$, $\alpha=2-c$, and $q_{n}=n$ for all $n\in \mathbb{N}$ in inequality (\ref{COPI}), we have
\begin{align*}
\displaystyle \sum_{n=1}^{\infty}S_n^{\alpha}\frac{|A_{n}-A_{n-1}|^{2}}{n} &\geq \frac{(\alpha-1)^{2}}{4}\displaystyle \sum_{n=1}^{\infty}\frac{n}{S_n^{2-\alpha}}|A_{n}|^{2},
\end{align*}or equivalently
\begin{align}{\label{COPIQNEQN}}
\displaystyle \sum_{n=1}^{\infty}S_n^{2-c}\frac{|A_{n}-A_{n-1}|^{2}}{n} &\geq \frac{(c-1)^{2}}{4}\displaystyle \sum_{n=1}^{\infty}\frac{n}{S_n^{c}}|A_{n}|^{2},
\end{align} where $S_n=\frac{n(n+1)}{2}$.\\
We then have two natural questions as below:
\begin{center}
\emph{(b) Is it possible to improve inequality (\ref{COPIQNEQN})?},
\emph{and if the answer is affirmative then} \\
\emph{(c) What are values of $c$ $(1<c\leq 2)$ for which improvement of (\ref{COPIQNEQN}) possible?}
\end{center}
Therefore, the aim of this present note is two folds. In one fold of this paper, we will answer question \emph{(a)} posed above, and establish a generalized improved discrete Hardy's inequality in one dimension. This extension not only includes the inequality (\ref{IMHI}) of Keller et al. \cite{MKR} but also contains the recent inequality (\ref{IMHIGN3}) of Krej\v{c}i\v{r}\'{i}k et al. (see Theorem 10, \cite{DKALFS}) as well as the inequality (\ref{GUPI}) of Gupta \cite{SGA}. In another fold, we will be concentrated on the answers to questions \emph{(b)}, and \emph{(c)}. We prove that improvement of Copson's inequality (\ref{COPIQNEQN}) is possible, and moreover, we can prove that it is true only when $c=\frac{3}{2}$.\\
The paper is organized as follows, in section 2 we prove an extended improved one-dimensional Hardy's inequality, and derive several consequences. In section 3, we will be dealing with Copson's inequality and its possible improvements and optimality of the weight sequence. In addition to these results, section 4, provides some fundamental properties of the sequence space $\Gamma_p$, $p>1$ which is created from the improved Hardy inequality (\ref{IMHI2}), and Copson inequality (\ref{IMCOP}).

\section{Generalization of Improved discrete Hardy's inequality}
We now proceed to present one of the main results of this paper. First, we begin with the following theorem.
\begin{thm}{\label{THMIDHI}}
Let $A_n$ be any sequence of real or complex numbers such that $A_n\in C_c(\mathbb{N}_0)$ with $A_{0}=0$ and $g=\{g_{n}\}_{n=1}^{\infty}$ be any strictly positive sequence of real numbers. Then the following inequality holds:
\begin{align}{\label{GIMHI}}
\displaystyle\sum_{n=1}^{\infty}w_{n}(\lambda,g )|A_n|^{2}\leq\displaystyle\sum_{n=1}^{\infty}\frac{|A_{n}-A_{n-1}|^{2}}{\lambda_{n}},
\end{align} where $\lambda=\{\lambda_n\}_{n\geq 1}$ such that $\lambda_{n}>0$, $n\in \mathbb{N}$, and the sequence $w_{n}(\lambda,g)$ is defined as below:
\begin{align*}
w_{n}(\lambda,g)& =\frac{1}{\lambda_{n}}+\frac{1}{\lambda_{n+1}}-\frac{g_{n-1}}{\lambda_{n}g_{n}}-\frac{g_{n+1}}{\lambda_{n+1}g_{n}}.
\end{align*}
Further, if there exists a sequence of elements $\gamma^{N}\in C_{c}(\mathbb{N}_{0})$ such that $\gamma^{N}\leq\gamma^{N+1}$ with $\gamma^{N}\rightarrow1$ as $N\rightarrow\infty$ pointwise, and
\begin{align}{\label{LIMCOND}}
\displaystyle\lim_{N\rightarrow\infty}\sum_{n=2}^{\infty}\frac{g_{n}g_{n-1}}{\lambda_{n}}|\gamma^{N}_{n}-\gamma^{N}_{n-1}|^{2}=0,
\end{align}
then $w_{n}(\lambda, g)$ is optimal.
\end{thm}
\begin{proof}
To prove the inequality (\ref{GIMHI}), we first choose $h_{n}=g_{n}-g_{n-1}$ so that $h_{n+1}=g_{n+1}-g_{n}$. Then $w_{n}(\lambda, g)$ can be re-written as follows:
\begin{center}
$w_{n}(\lambda,g)=\frac{h_{n}}{\lambda_{n}g_{n}}-\frac{h_{n+1}}{\lambda_{n+1}g_{n}}$.
\end{center}
We observe that
\begin{align*}
\displaystyle\sum_{n=1}^{\infty}w_{n}(\lambda,g)|A_{n}|^{2}
& =\displaystyle\sum_{n=1}^{\infty}\Big(\frac{h_{n}}{\lambda_{n}g_{n}}-\frac{h_{n+1}}{\lambda_{n+1}g_{n}}\Big)|A_{n}|^{2}\\
&=\displaystyle\sum_{n=1}^{\infty}\frac{h_{n}}{\lambda_{n}g_{n}}|A_{n}|^{2}-\displaystyle\sum_{n=2}^{\infty}\frac{h_{n}}{\lambda_{n}g_{n-1}}|A_{n-1}|^{2} \\
&=\displaystyle\sum_{n=1}^{\infty}\frac{h_{n}}{\lambda_{n}}\Big(\frac{|A_{n}|^{2}}{g_{n}}-\frac{|A_{n-1}|^{2}}{g_{n-1}}\Big),
\end{align*}where it is assumed that the term $\frac{|A_0|^2}{\lambda_{1}g_{0}}$ as zero. Further by assuming the term $\sqrt{\frac{g_{1}}{g_{0}}}\frac{A_{0}}{\sqrt\lambda_{1}}$ as zero, the following computation gives
\begin{align*}
&\displaystyle\sum_{n=1}^{\infty}\frac{|A_{n}-A_{n-1}|^{2}}{\lambda_{n}}-\displaystyle\sum_{n=1}^{\infty}w_{n}(\lambda,g)|A_n|^{2}\\
& =\displaystyle\sum_{n=1}^{\infty}\Big[\Big|\frac{A_{n}}{\sqrt\lambda_{n}}-\frac{A_{n-1}}{\sqrt\lambda_{n}}\Big|^{2}-\frac{h_{n}}{\lambda_{n}}\Big(\frac{|A_{n}|^{2}}{g_{n}}-\frac {|A_{n-1}|^{2}}{g_{n-1}}\Big)\Big]\\
& =\displaystyle\sum_{n=1}^{\infty}\Big[\frac{|A_{n}|^{2}}{\lambda_{n}}+\frac{|A_{n-1}|^{2}}{\lambda_{n}}-\frac{2\mathfrak{R}(\bar{A}_{n}A_{n-1})}{\lambda_{n}}-\frac{h_{n}}
{\lambda_{n}}\Big(\frac{|A_{n}|^{2}}{g_{n}}-\frac {|A_{n-1}|^{2}}{g_{n-1}}\Big)\Big]\\
&=\displaystyle\sum_{n=1}^{\infty}\Big[\Big(1-\frac{h_{n}}{g_{n}}\Big)\frac{|A_{n}|^{2}}{\lambda_{n}}+\Big(1+\frac {h_{n}}{g_{n-1}}\Big)\frac{|A_{n-1}|^{2}}{\lambda_{n}}-\frac{2\mathfrak{R}(\bar{A}_{n}A_{n-1})}{\lambda_{n}}\Big]
\end{align*}
\begin{align*}
&=\displaystyle\sum_{n=2}^{\infty}\Big|\sqrt{(1-\frac{h_{n}}{g_{n}}})\frac{A_{n}}{\sqrt{\lambda_{n}}}-\sqrt{(1+\frac{h_{n}}{g_{n-1}})}\frac{A_{n-1}}
{\sqrt\lambda_{n}}\Big|^{2}\geq0.
\end{align*}
Therefore, the above discussion gives the following identity
\begin{align}{\label{THM1IDN}}
&\displaystyle\sum_{n=2}^{\infty}\Big|\sqrt{(1-\frac{h_{n}}{g_{n}})}\frac{A_{n}}{\sqrt{\lambda_{n}}}-\sqrt{(1+\frac{h_{n}}{g_{n-1}})}
\frac{A_{n-1}}{\sqrt{\lambda_{n}}}\Big|^{2}+\displaystyle\sum_{n=1}^{\infty}\frac{h_{n}}{\lambda_{n}}\Big(\frac{|A_{n}|^{2}}{g_{n}}-\frac {|A_{n-1}|^{2}}{g_{n-1}}\Big) \nonumber\\
&=\displaystyle\sum_{n=1}^{\infty}\Big|\frac{A_{n}}{\sqrt\lambda_{n}}-\frac{A_{n-1}}{\sqrt\lambda_{n}}\Big|^{2},
\end{align}
which proves the desired inequality (\ref{GIMHI}). This completes the proof of the theorem.\\
We now prove the optimality of the sequence $w_n$. Let $\widetilde{w}_{n}$ be a sequence satisfying $\widetilde{w}_{n}\geq w_{n}$. Then using the above identity (\ref{THM1IDN}), we get
\begin{align}{\label{THM1OPTI}}
0& \leq\displaystyle\sum_{n=1}^{\infty}(\widetilde{w}_{n}-w_{n})|A_{n}|^{2}\leq\displaystyle\sum_{n=2}^{\infty}\Big|\sqrt{ \frac{g_{n-1}}{g_{n}}}\frac{A_{n}}{\sqrt\lambda_{n}}-\sqrt{\frac{g_{n}}{g_{n-1}}}\frac{A_{n-1}}{\sqrt\lambda_{n}}\Big|^{2}
\end{align}
Now choose a sequence $A_{n}=g_{n}\gamma^{N}_{n}$ in above inequality (\ref{THM1OPTI}), where $\gamma^{N}\in C_{c}(\mathbb{N}_{0})$ satisfying the given assumptions,
we get
\begin{center}
$0\leq\displaystyle\sum_{n=1}^{\infty}(\widetilde{w}_{n}-w_{n})|g_{n}\gamma^{N}_{n}|^{2}\leq\displaystyle\sum_{n=2}^{\infty}
\frac{g_{n}g_{n-1}}{\lambda_{n}}|\gamma^{N}_{n}-\gamma^{N}_{n-1}|^{2}$.
\end{center}
Since the sequence $\{g_n\}$ is strictly positive, $\gamma^{N}\leq\gamma^{N+1}$ and $\gamma^{N}\rightarrow1$ as $N\rightarrow\infty$ pointwise, so by monotone convergence theorem and (\ref{LIMCOND}), we obtain
\begin{center}
$\displaystyle\sum_{n=1}^{\infty}(\widetilde{w}_{n}-w_{n})g^{2}_{n}=0$.
\end{center}
Hence $\widetilde{w}_{n}=w_{n}$, which establishes the optimality of $w_{n}$.
\end{proof}

An immediate consequence of the Theorem \ref{THMIDHI} is the recent result of Krej\v{c}i\v{r}\'{i}k et al. (see Theorem 10, \cite{DKALFS})
\begin{cor}(Theorem 10, \cite{DKALFS})
If one chooses $\lambda_{n}=1$, $n\in \mathbb{N}$, and $\{g_n\}$ is a positive sequence of real numbers then
\begin{center}
$\displaystyle\sum_{n=1}^{\infty}w_{n}(g)|A_{n}|^2\leq\displaystyle\sum_{n=1}^{\infty}|A_{n}-A_{n-1}|^{2}$,
\end{center}
where the sequence $w_{n}=2-\frac{g_{n-1}}{g_{n}}-\frac{g_{n+1}}{g_{n}}$ is optimal if
\begin{align*}
\displaystyle\lim_{N\rightarrow\infty}\sum_{n=2}^{\infty}g_{n}g_{n-1}|\gamma^{N}_{n}-\gamma^{N}_{n-1}|^{2}=0.
\end{align*}
\end{cor}
The Theorem \ref{THMIDHI} also contains the inequality derived by Gupta \cite{SGA} as follows.
\begin{cor}(Gupta \cite{SGA})
If one chooses $\lambda_{n}=\frac{1}{n^{\alpha}}$, $\alpha\in \mathbb{R}$, and $g_n=n^{\beta}$ then
\begin{center}
$\displaystyle\sum_{n=1}^{\infty}w_{n}(\alpha,\beta)|A_{n}|^{2}$$\leq\displaystyle\sum_{n=1}^{\infty}|A_{n}-A_{n-1}|^{2}n^{\alpha}$,
\end{center}
where $\beta\in \mathbb{R}$ and
\begin{align*}
w_{n}(\alpha,\beta)&=n^{\alpha}\Big[1+(1+\frac{1}{n})^{\alpha}-(1-\frac{1}{n})^{\beta}-(1+\frac{1}{n})^{\alpha+\beta}\Big].
\end{align*}
\end{cor}
\begin{Remark}
We remark that if $w_{n}(\lambda,\beta)$ is expanded in a series then we have
\begin{align*}
&w_{n}(\lambda,\beta)\\
& =\frac{1}{\lambda_{n}}+\frac{1}{\lambda_{n+1}}-\frac{1}
{\lambda_{n}}\Big(1-\frac{1}{n}\Big)^{\beta}-\frac{1}{\lambda_{n+1}}\Big(1+\frac{1}{n}\Big)^{\beta}\\
&=\Big(\frac{1}{\lambda_{n}}-\frac{1}{\lambda_{n+1}}\Big)\frac{\beta}{n}+\frac{\beta(1-\beta)}{2}\frac{1}{n^{2}}\Big(\frac{1}{\lambda_{n}}+\frac{1}{\lambda_{n+1}}\Big)+ \frac{\beta(\beta-1)(\beta-2)}{6}\frac{1}{n^{3}}\Big(\frac{1}{\lambda_{n}}-\frac{1}{\lambda_{n+1}}\Big)\\
& \hspace{0.15cm}
+\frac{\beta(1-\beta)(\beta-2)(\beta-3)}{24}\frac{1}{n^{4}}\Big(\frac{1}{\lambda_{n}}+\frac{1}{\lambda_{n+1}}\Big)+\cdots\cdots
\end{align*}
Hence if $\beta\in(0,1)$ and the sequence $\{\lambda_{n}\}$ is non-decreasing then one gets
\begin{center}
$w_{n}(\lambda,\beta)>\frac{\beta(1-\beta)}{2}\frac{1}{n^{2}}(\frac{1}{\lambda_{n}}+\frac{1}{\lambda_{n+1}})$.
\end{center}
\end{Remark}
Therefore for any $\beta\in(0,1)$ and $\lambda_{n+1}\geq\lambda_{n}$, $n\in\mathbb{N}$ one can asserts the following:
\begin{align}{\label{RBETA}}
\displaystyle\sum_{n=1}^{\infty}\Big(\frac{1}{\lambda_{n}}+\frac{1}{\lambda_{n+1}}\Big)\frac{\beta(1-\beta)}{2}\frac{1}{n^{2}}|A_{n}|^{2}
&<\displaystyle\sum_{n=1}^{\infty}w_{n}(\lambda,\beta) |A_{n}|^{2}\leq\displaystyle\sum_{n=1}^{\infty}\frac{|A_{n}-A_{n-1}|^{2}}{\lambda_{n}}.
\end{align}

\begin{Remark}{\label{RGIMHI}}
In particular, if $\beta=\frac{1}{2}$ is chosen in $w_{n}(\lambda,\beta)$, then we get
$$w_{n}(\lambda, \frac{1}{2})=\frac{1}{\lambda_{n}}+\frac{1}{\lambda_{n+1}}-\frac{1}
{\lambda_{n}}\sqrt{(1-\frac{1}{n})}-\frac{1}{\lambda_{n+1}}\sqrt{(1+\frac{1}{n})}.$$
Moreover, we note that
\begin{center}
$w_{n}(\lambda, \frac{1}{2})>\frac{1}{8n^{2}}(\frac{1}{\lambda_{n}}+\frac{1}{\lambda_{n+1}})$
\end{center}
is true for all non-decreasing sequence $\{\lambda_{n}\}$.
\end{Remark}

\begin{cor}
It is to be noted that for $\lambda_{n}=1$, $n\in \mathbb{N}$ and $\beta=\frac{1}{2}$ then we get the following inequality due to Keller et al. \cite{MKR}:
\begin{center}
$\displaystyle\sum_{n=1}^{\infty}\frac{1}{4n^{2}}|A_{n}|^{2} < \displaystyle\sum_{n=1}^{\infty}w_{n}|A_{n}|^{2} \leq \displaystyle\sum_{n=1}^{\infty}|A_{n}-A_{n-1}|^{2}$
\end{center}
where $ w_{n}\equiv w_{n}(\{1\}, \frac{1}{2})=2-\sqrt{(1-\frac{1}{n})}-\sqrt{(1+\frac{1}{n})}$.
\end{cor}
Therefore, from the above discussions, we have the following theorem.
\begin{thm}{\label{THMOPTI}} Let $g_{n}=n^{\beta}$, $\beta\in (0, 1)$ and $\lambda=\{\lambda_{n}\}$ be a non-decreasing sequence of real numbers such that $\lambda_n>0$. Then for any sequence $A=\{A_n\}$  of real or complex numbers with $A\in C_c(\mathbb{N}_0)$ such that $A_{0}=0$ the following inequality holds:
\begin{align}
\displaystyle\sum_{n=1}^{\infty}\Big(\frac{1}{\lambda_{n}}+\frac{1}{\lambda_{n+1}}\Big)\frac{\beta(1-\beta)}{2n^{2}}|A_{n}|^{2}
& < \displaystyle\sum_{n=1}^{\infty}w_{n}(\lambda,\beta)|A_{n}|^{2} \leq \displaystyle\sum_{n=1}^{\infty}\frac{|A_{n}-A_{n-1}|^{2}}{\lambda_{n}}.
\end{align}Also for $\lambda_n\geq 1$, $n\in \mathbb{N}$, and $\beta\in (0,\frac{1}{2}]$ the weight sequence $w_{n}(\lambda, \beta)$ is optimal.
\end{thm}
\begin{proof}
The proof of the theorem is a combination of Theorem \ref{THMIDHI} and inequality (\ref{RBETA}). We now only establish the optimality of the weight sequence $w_{n}(\lambda,\beta)$ here, and for this we follow the  Krej\v{c}i\v{r}\'{i}k and \v{S}tampach \cite{DKK} approach. Note that for $\beta\in (0,\frac{1}{2}]$ the identity (\ref{THM1IDN}) gives
\begin{align}{\label{GHI}}
\displaystyle\sum_{n=1}^{\infty}\frac{|A_{n}-A_{n-1}|^{2}}{\lambda_{n}}
&=\displaystyle\sum_{n=2}^{\infty}\Big|\Big(\frac{n-1}{n}\Big)^{\frac{\beta}{2}}\frac{A_n}{\sqrt{\lambda_{n}}}-\Big(\frac{n}{n-1}\Big)^{\frac{\beta}{2}}\frac{A_{n-1}}{\sqrt{\lambda_{n}}}\Big|^{2}
+\displaystyle\sum_{n=1}^{\infty}w_{n}(\lambda,\beta)|A_{n}|^{2}.
\end{align}
Suppose $\widetilde{w_{n}}(\lambda,\beta)$ is a sequence such that $\widetilde{w_{n}}(\lambda,\beta)\geq w_{n}(\lambda,\beta)$ holds for all $n\in \mathbb{N}$ and inequality (\ref{GHI}) holds for the sequence $\widetilde{w_{n}}(\lambda,\beta)$. Then for all $A\in C_c(\mathbb{N}_0)$, we have the following:
\begin{align}{\label{WNTILDE}}
\displaystyle\sum_{n=2}^{\infty}
\Big|\Big(\frac{n-1}{n}\Big)^{\frac{\beta}{2}}\frac{A_n}{\sqrt{\lambda_{n}}}-\Big(\frac{n}{n-1}\Big)^{\frac{\beta}{2}}\frac{A_{n-1}}{\sqrt{\lambda_{n}}}\Big|^{2}& \geq \displaystyle\sum_{n=1}^{\infty}(\widetilde{w_{n}}(\lambda,\beta)-w_{n}(\lambda,\beta))|A_{n}|^{2} \geq 0.
\end{align}
Since inequality (\ref{WNTILDE}) vanishes when $A_n=n^{\beta}$, $ \beta\in (0,\frac{1}{2}]$ which is not a finitely supported sequence, so we redefine the sequence $A^{N}_{n}$, $N\geq2$ as $A^{N}_{n}=\gamma^{N}_{n}n^{\beta}$ where $\gamma^{N}_{n}$ defined as below:
\begin{align*}
\gamma^{N}_{n} &= \left\{
\begin{array}{lll}
    1 & \quad \mbox{if~~} n<N,\\
    \frac {2\log N-\sqrt\lambda_{n}\log n}{\log N} & \quad \mbox{if~~} N\leq n\leq N^{2}\\
    0 & \quad \mbox{if~~} n>N^{2}.
\end{array}\right.
\end{align*}
Now
\begin{align*}
&\displaystyle\sum_{n=2}^{\infty}\Big|\Big(\frac{n-1}{n}\Big)^{\frac{\beta}{2}}\frac{A_n^N}{\sqrt{\lambda_{n}}}-\Big(\frac{n}{n-1}\Big)^{\frac{\beta}{2}}\frac{A_{n-1}^N}{\sqrt{\lambda_{n}}}\Big|^{2}\\
&=\displaystyle\sum_{n=2}^{\infty}\frac{1}{\lambda_n}\big(n(n-1)\big)^{\beta}\Big|\gamma^{N}_{n}-\gamma^{N}_{n-1}\Big|^{2}\\
&=\frac{1}{(\log N)^{2}}\displaystyle\sum_{n=N+1}^{N^{2}}\frac{1}{\lambda_{n}}\big(n(n-1)\big)^{\beta}\Big|\sqrt{\lambda_{n-1}}\log(n-1)-\sqrt{\lambda_{n}}\log n\Big|^{2}\\
&\leq \frac{1}{(\log N)^{2}}\displaystyle\sum_{n=N+1}^{N^{2}}\sqrt{n(n-1)}\Big(\log\Big(\frac{n}{n-1}\Big)\Big)^2~~~ (\mbox{since~} \{\lambda_{n}\} \mbox{~is non-decreasing})\\
& \leq\frac{1}{(\log N)^{2}}\displaystyle\sum_{n=N+1}^{N^{2}}\frac{\sqrt{n(n-1)}}{(n-1)^{2}} \leq \frac{1}{(\log N)^{2}}\displaystyle\sum_{n=N+1}^{N^{2}}\Big\{\frac{1}{(n-1)}+\frac{1}{2(n-1)^2}\Big\}
\end{align*}
\begin{align*}
& \leq \frac{1}{(\log N)^{2}} \displaystyle\int_{N}^{N^2}\Big\{\frac{1}{(n-1)}+\frac{1}{2(n-1)^2}\Big\}dn\\
& \leq \frac{1}{\log N} + \frac{N}{2(N^2-1)(\log N)^{2}},
\end{align*}
which tends to $0$ as $N\rightarrow\infty$. Hence from (\ref{WNTILDE}), we get
\begin{center}
$\displaystyle\sum_{n=1}^{\infty}(\widetilde{w_{n}}(\lambda,\beta)-w_{n}(\lambda,\beta))|A_{n}|^{2}=0$.
\end{center}
Therefore, $\widetilde{w_{n}}(\lambda,\beta)= w_{n}(\lambda,\beta)$ holds for all $n\in \mathbb{N}$, that is the sequence $w_{n}(\lambda,\beta)$, $\beta\in(0,\frac{1}{2}]$ is optimal. This completes proof of the theorem.
\end{proof}
\section{Improvement of Copson's inequality}
In this section, we present an improvement of Copson's inequality (\ref{COPIQNEQN}). To achieve it, first we need to establish the following theorem, the statement of which is given below.
\begin{thm}
Let $\{a_n\}$ be any sequence of real or complex numbers such that $A_{n}=\displaystyle\sum_{k=1}^{n}ka_{k}$ with $A_{0}=0$. Then for $1<c\leq 2$, we have
\begin{align}{\label{ineqimvd}}
\displaystyle\sum_{n=1}^{\infty}V_{n}|A_n|^{2}&\leq\displaystyle\sum_{n=1}^{\infty}
S_n^{2-c}\frac{|A_{n}-A_{n-1}|^{2}}{n},
\end{align}where the sequence $V_{n}$ is defined as below:
\begin{align*}
V_{n}& =\frac{S_n^{2-c}}{n}+\frac{S_{n+1}^{2-c}}{n+1}-\frac{S_n^{2-c}}{n}\sqrt{1-\frac{1}{n}}-\frac{S_{n+1}^{2-c}}{n+1}\sqrt{1+\frac{1}{n}},
\end{align*} with $S_n=\frac{n(n+1)}{2}$.
\end{thm}
\begin{proof}
Let us denote $\tau_{n}=\frac{\sqrt n-\sqrt{ (n-1)}}{n}$ so that $\tau_{n+1}=\frac{\sqrt{(n+1)}-\sqrt n}{n+1}$. Then $V_{n}$ reduces to the following form:
\begin{center}
$V_{n}= \frac{1}{\sqrt n}\Big(\tau_{n}S_n^{2-c}-\tau_{n+1}S_{n+1}^{2-c}\Big)$.
\end{center}
Thus using the above expression for $V_n$, one can have the following identity.
\begin{align*}
\displaystyle\sum_{n=1}^{\infty}V_{n}|A_{n}|^{2}
& =\displaystyle\sum_{n=1}^{\infty}\Big[\frac{\tau_{n}S_n^{2-c}}{\sqrt n}-\frac{\tau_{n+1}S_{n+1}^{2-c}}{\sqrt n}\Big]|A_{n}|^{2}\\
&=\displaystyle\sum_{n=1}^{\infty}\frac{\tau_{n}}{\sqrt n}S_n^{2-c}|A_{n}|^{2}-\displaystyle\sum_{n=2}^{\infty}
\frac{\tau_{n}}{\sqrt{(n-1)}}S_n^{2-c}|A_{n-1|^{2}}\\
&=\displaystyle\sum_{n=1}^{\infty}\tau_{n}\Big(\frac{|A_{n}|^{2}}{\sqrt n}-\frac{|A_{n-1}|^{2}}{\sqrt{(n-1)}}\Big)S_n^{2-c}.
\end{align*}
Now let us compute the difference between the terms in both sides of inequality (\ref{ineqimvd}), that is
\begin{align}{\label{COPMAINREI}}
&\displaystyle\sum_{n=1}^{\infty}
\frac{|A_{n}-A_{n-1}|^{2}}{n}S_n^{2-c}-\displaystyle\sum_{n=1}^{\infty}V_{n}|A_n|^{2} \nonumber\\
& =\displaystyle\sum_{n=1}^{\infty}\Big[\Big|\frac{A_{n}}{\sqrt n}-\frac{A_{n-1}}{\sqrt n}\Big|^{2}S_n^{2-c}-
\tau_{n}\Big(\frac{|A_{n}|^{2}}{\sqrt n}-\frac {|A_{n-1}|^{2}}{\sqrt{(n-1)}}\Big)S_n^{2-c}\Big] \nonumber\\
& =\displaystyle\sum_{n=1}^{\infty}\Big[\frac{|A_{n}|^{2}}{n}+\frac{|A_{n-1}|^{2}}{n}-\frac{2\mathfrak{R}(\bar{A}_{n}A_{n-1})}{n}-
\tau_{n}\Big(\frac{|A_{n}|^{2}}{\sqrt n}-\frac {|A_{n-1}|^{2}}{\sqrt{(n-1)}}\Big)\Big]S_n^{2-c} \nonumber\\
&=\displaystyle\sum_{n=1}^{\infty}\Big[\Big(1-\frac{n\tau_{n}}{\sqrt n}\Big)
\frac{|A_{n}|^{2}}{n}+\Big(1+\frac {n\tau_{n}}{\sqrt{(n-1)}}\Big)\frac{|A_{n-1}|^{2}}{n}-\frac{2\mathfrak{R}(\bar{A}_{n}A_{n-1})}{n}\Big]S_n^{2-c} \nonumber\\
&=\displaystyle\sum_{n=1}^{\infty}\Big|\sqrt{(1-\frac{n\tau_{n}}{\sqrt n})}\frac{A_{n}}{\sqrt n}
-\sqrt{(1+\frac{n\tau_{n}}{\sqrt{(n-1)}})}\frac{A_{n-1}}{\sqrt{n}}
\Big|^{2}S_n^{2-c} \nonumber \\
&=\displaystyle\sum_{n=1}^{\infty}\Big|\frac{A_{n}}{\sqrt n}\sqrt[4]{\frac{n-1}{n}}
-\frac{A_{n-1}}{\sqrt n}\sqrt[4]{\frac{n}{n-1}}\Big|^{2} S_n^{2-c}\\
& \geq 0. \nonumber
\end{align}
Hence the desired inequality (\ref{ineqimvd}) follows easily.
\end{proof}

Next we present few results with the aim of establishing improvement of Copson's inequality (\ref{COPIQNEQN}).
\begin{Lemma}{\label{COPIML1}}
If $2\geq c\geq\frac{3}{2}$ then $\forall n\in\mathbb{N}$ we get
\begin{align*}
\frac{S_n^{2-c}}{n}& >\frac{S_{n+1}^{2-c}}{n+1}.
\end{align*}
\end{Lemma}
\begin{proof}We know that for $n\in \mathbb{N}$, $(n+1)^{2}>n^{2}+2n$. Hence we have
$\Big(\frac{n+1}{n}\Big)^{2}>\frac{n+2}{n}$.\\
Therefore for each $c\geq\frac{3}{2}$, we get
\begin{align*}
\frac{n+1}{n}& >\Big(\frac{n+2}{n}\Big)^{\frac{1}{2}}\geq\Big(\frac{n+2}{n}\Big)^{2-c}.
\end{align*}
Now an easy computation gives
\begin{align*}
\frac{S_n^{2-c}}{n}-\frac{S_{n+1}^{2-c}}{n+1}
& > \Big(\frac{n+2}{n}\Big)^{2-c}\frac{S_n^{2-c}}{n+1}-\frac{S_{n+1}^{2-c}}{n+1}=0,
\end{align*}which proves the desired inequality.
\end{proof}

\begin{Lemma}{\label{COPIML2}}
For any $2\geq c \geq \frac{3}{2}$, we get
\begin{align*}
V_{n}&>\Big(\frac{S_n^{2-c}}{n}+\frac{S_{n+1}^{2-c}}{n+1}\Big)\frac{1}{8n^{2}}.
\end{align*}
\end{Lemma}
\begin{proof}
Using the Lemma \ref{COPIML1}, one can deduce the following inequality:
\begin{align*}
V_{n}&=\frac{S_n^{2-c}}{n}+\frac{S_{n+1}^{2-c}}{n+1}-\frac{S_n^{2-c}}{n}\sqrt{(1-\frac{1}{n})}-\frac{S_{n+1}^{2-c}}{n+1}\sqrt{(1+\frac{1}{n})}\\
&=\Big(\frac{S_n^{2-c}}{n}-\frac{S_{n+1}^{2-c}}{n+1}\Big)\frac{1}{2n}+\Big(\frac{S_n^{2-c}}{n}+\frac{S_{n+1}^{2-c}}{n+1}\Big)\frac{1}{8n^{2}}+
\Big(\frac{S_n^{2-c}}{n}-\frac{S_{n+1}^{2-c}}{n+1}\Big)\frac{1}{16n^{3}}+\Big(\frac{S_n^{2-c}}{n}\\
& \hspace{1cm}+\frac{S_{n+1}^{2-c}}{n+1}\Big)\frac{5}{128n^{4}}+\cdots\cdots\\
&>\Big(\frac{S_n^{2-c}}{n}+\frac{S_{n+1}^{2-c}}{n+1}\Big)\frac{1}{8n^{2}}+\Big(\frac{S_n^{2-c}}{n}+\frac{S_{n+1}^{2-c}}{n+1}\Big)\frac{5}{128n^{4}}+\cdots\cdots
\end{align*}Hence we have desired inequality.
\end{proof}

\begin{Lemma}{\label{COPIML3}}
Let $n\in \mathbb{N}$. Then for any $1<c\leq\frac{3}{2}$, we have
\begin{align*}
\frac{1}{8n^{2}}\Big[\frac{S_n^{2-c}}{n}+\frac{S_{n+1}^{2-c}}{n+1}\Big]& > \frac{(c-1)^{2}}{4}\frac{n}{S_n^{c}}.
\end{align*}
\end{Lemma}
\begin{proof}
Let $n\in \mathbb{N}$ and $c>1$. Then we have the following steps
\begin{align*}
& c\leq\frac{3}{2}\\
\Rightarrow &8(c-1)^{2}\leq2<(1+\frac{1}{n})^{2}+(1+\frac{1}{n})(1+\frac{2}{n})^{2-c}\\
\Rightarrow & 8(c-1)^{2}<\frac{(n+1)^{2}}{n^{2}}+(n+1)(n+2)^{2-c}n^{c-3}\\
\Rightarrow &\frac {n(c-1)^{2}}{4}<\Big[\frac{(n+1)^{2}n}{4}+\frac{(n+1)(n+2)^{2-c}}{4}n^{c}\Big]\frac{1}{8n^{2}}=\Big[\frac{S_n^{2}}{n}+\frac{S_n^{c}S_{n+1}^{2-c}}{n+1}\Big]\frac{1}{8n^{2}}\\
\Rightarrow&\frac{(c-1)^{2}}{4}\frac{n}{S_n^{c}}<\frac{1}{8n^{2}}\Big[\frac{S_n^{2-c}}{n}+\frac{S_{n+1}^{2-c}}{n+1}\Big].
\end{align*}
Hence the desired result.
\end{proof}

\begin{Lemma}{\label{COPIML4}}
Suppose that $n\in \mathbb{N}$. Then we have
\begin{align*}
V_n& >\frac{1}{8n^{2}}\Big[\frac{\sqrt{S_n}}{n}+\frac{\sqrt{S_{n+1}}}{n+1}\Big]> \frac{1}{16}\frac{n}{{S_n\sqrt{S_n}}}.
\end{align*}
\end{Lemma}
\begin{proof}
Since Lemma \ref{COPIML2} holds for all $\frac{3}{2}\leq c \leq 2$, and Lemma \ref{COPIML3} satisfied for $1<c\leq\frac{3}{2}$, so at the intersecting point that is for $c=\frac{3}{2}$ both the Lemmas hold. Therefore, by combining both these Lemma for the case when $c=\frac{3}{2}$, we get the desired inequality.
\end{proof}
Therefore using the Lemma \ref{COPIML4}, we get an improvement of the inequality (\ref{COPIQNEQN}) for the case when $c=\frac{3}{2}$ as follows:
\begin{align}{\label{IMCOP}}
\displaystyle \sum_{n=1}^{\infty}\frac{1}{16}\frac{n}{{S_n\sqrt{S_n}}}|A_{n}|^{2} & <\displaystyle\sum_{n=1}^{\infty}V_{n}|A_{n}|^{2}\leq\displaystyle \sum_{n=1}^{\infty}{\sqrt{S_n}}\frac{|A_{n}-A_{n-1}|^{2}}{n}.
\end{align}
Now we prove the optimality of the sequence $V_n$. For this let $\widetilde{V_{n}}$ is a sequence such that $\widetilde{V_{n}}\geq V_{n}$ holds for all $n\in \mathbb{N}$ and inequality (\ref{COPMAINREI}) holds for the sequence $\widetilde{V_{n}}$. Then we have the following:
\begin{align}{\label{COPREM}}
\displaystyle\sum_{n=2}^{\infty}
\Big|\sqrt[4]{\frac{n-1}{n}}\frac{A_n}{\sqrt n}-\sqrt[4]{\frac{n}{n-1}}\frac{A_{n-1}}{{\sqrt n}}\Big|^{2}\sqrt{S_n} &\geq \displaystyle\sum_{n=1}^{\infty}(\widetilde{V_{n}}-V_{n})|A_{n}|^{2} \geq 0.
\end{align}
We now redefine the sequence $A^{N}_{n}$, $N\geq2$ as $A^{N}_{n}=\gamma^{N}_{n}\sqrt{n}$ where $\gamma^{N}_{n}$ defined as below:
\begin{align*}
\gamma^{N}_{n} &= \left\{
\begin{array}{lll}
    1 & \quad \mbox{if~~} n<N,\\
    \frac {2\log N- \sqrt[4]{\frac{2n}{(n+1)}}\log n}{\log N} & \quad \mbox{if~~} N\leq n\leq N^{2}\\
    0 & \quad \mbox{if~~} n>N^{2}.
\end{array}\right.
\end{align*}
Now using the above sequence, and Theorem \ref{THMOPTI} at the last stage, we obtain
\begin{align*}
&\displaystyle\sum_{n=2}^{\infty}\Big|\sqrt[4]{\frac{n-1}{n}}A_{n}^N-\sqrt[4]{\frac{n}{n-1}}A_{n-1}^N\Big|^{2}{\sqrt\frac{(n+1)}{2n}}\\
&=\displaystyle\sum_{n=2}^{\infty}\Big|\sqrt[4]{\frac{n-1}{n}}\gamma^{N}_{n}\sqrt{n}-\sqrt[4]{\frac{n}{n-1}}\gamma^{N}_{n-1}\sqrt{(n-1)}\Big|^{2}{\sqrt\frac{(n+1)}{2n}}\\
&=\displaystyle\sum_{n=2}^{\infty}\sqrt{n(n-1)}\Big|\gamma^{N}_{n}-\gamma^{N}_{n-1}\Big|^{2}{\sqrt\frac{(n+1)}{2n}}\\
&=\frac{1}{(\log N)^{2}}\displaystyle\sum_{n=N+1}^{N^{2}}\sqrt{n(n-1)}\Big|\sqrt[4]{\frac{2(n-1)}{n}}\log(n-1)-\sqrt[4]{\frac{2n}{n+1}}\log n\Big|^{2}{\sqrt\frac{(n+1)}{2n}}\\
&\leq\frac{1}{(\log N)^{2}}\displaystyle\sum_{n=N+1}^{N^{2}}\sqrt{n(n-1)}\Big|\sqrt[4]{\frac{2(n-1)}{n}}\log(n-1)-\sqrt[4]{\frac{2(n-1)}{n}}\log(n)\Big|^{2}{\sqrt\frac{(n+1)}{2n}}\\
& \hspace{0.5cm}\big(\mbox{since~} \sqrt[4]{\frac{2(n-1)}{n}}\leq \sqrt[4]{\frac{2n}{n+1}} ~~\&~~ \sqrt {1-\frac{1}{n^{2}}}\leq1\big)\\
&\leq \frac{1}{(\log N)^{2}}\displaystyle\sum_{n=N+1}^{N^{2}}\sqrt{n(n-1)}(\log(\frac{n}{n-1}))^2\\
& \leq\frac{1}{(\log N)^{2}}\displaystyle\sum_{n=N+1}^{N^{2}}\frac{\sqrt{n(n-1)}}{(n-1)^{2}}\\
& \leq \frac{1}{\log N} + \frac{N}{2(N^2-1)(\log N)^{2}},
\end{align*}
which tends to $0$ as $N\rightarrow\infty$. Hence from (\ref{COPREM}), we get
\begin{center}
$\displaystyle\sum_{n=1}^{\infty}(\widetilde{V}_n-V_n)|A_{n}|^{2}=0$.
\end{center}
Therefore, $\widetilde{V}_n= V_n$ holds for all $n\in \mathbb{N}$, hence the sequence $V_{n}$ is optimal. This finishes proof of the optimality.

\section{Fundamental structure of sequence space $\Gamma_p$}
In an unpublished preprint of Fischer et al. \cite{FFR}, an improved Hardy's inequality is obtained for a general $p>1$. Indeed, for $p>1$ it was proved that
\begin{align}{\label{IHGP}}
\displaystyle\sum_{n=1}^{\infty}|A_n-A_{n-1}|^p \geq \sum_{n=1}^{\infty}w_n|A_n|^p>\Big(\frac{p-1}{p}\Big)^p\sum_{n=1}^{\infty}\frac{1}{n^p}|A_n|^p,
\end{align}where $w_n=\Big(1-\Big(\frac{n-1}{n}\Big)^{\frac{p-1}{p}}\Big)^{p-1}-\Big(\Big(\frac{n+1}{n}\Big)^{\frac{p-1}{p}}-1\Big)^{p-1}$, and $A_n=a_1+a_2+\ldots+a_n$. Note that in the case of $p=2$, one gets the inequality (\ref{IMHI2}), an improved version of discrete Hardy's inequality by Keller et al. \cite{MKR}. Therefore, it is natural to study some fundamental properties of the underlying space $\Gamma_p$, which is defined as follows. Let $s$ be denoted as the space of all real or complex sequences, $\{\gamma_n\}$, and $\{q_n\}$ be two sequences of real numbers such that $\gamma_n>0$, and $q_n>0$ for each $n\in \mathbb{N}$. Then for $1<p<\infty$, we define
\begin{align}{\label{GPSEQ}}
\Gamma_p&= \Big\{x=\{a_n\}\in s: \displaystyle\sum_{n=1}^{\infty} \gamma_n\Big|\sum_{k=1}^{n}q_ka_k\Big|^p < \infty\Big\},
\end{align} and when $\gamma_n=\frac{1}{n^p}$ with $q_n=1$ for all $n\in \mathbb{N}$, we denote $\Gamma_p$ by $X_p$, where
\begin{align}{\label{XPSEQ}}
X_p&= \Big\{x=\{a_n\}\in s: \displaystyle\sum_{n=1}^{\infty} \frac{1}{n^p}\Big|\sum_{k=1}^{n}a_k\Big|^p < \infty\Big\}.
\end{align}
The sequence space $X_p$ is known as Ces\`{a}ro sequence spaces of non-absolute type and studied in \cite{PNNGLEE}. Note that when we choose $\gamma_n=w_n$, and $q_n=1$ for all $n\in \mathbb{N}$ in (\ref{GPSEQ}), then $\Gamma_p$ reduces to the sequence space $W_p$ (say), which is directly involved in the inequality (\ref{IHGP}). In case when $p=2$, sequence space $W_2$ is connected with (\ref{IMHI2}). Also for $\gamma_n=V_n$, $p=2$, and $q_n=n$ for all $n\in \mathbb{N}$, we get a sequence space, which is demonstrated in (\ref{IMCOP}).
It is easy to verify that the sequence space $\Gamma_p$ defined as in (\ref{GPSEQ}) is a normed linear space equipped with the norm functional $\|\cdot\|_{\Gamma_p}$ defined as below:
\begin{center}
$\|x\|_{\Gamma_p}=\Big(\displaystyle\sum_{n=1}^{\infty}\gamma_{n} \Big|\displaystyle\sum_{k=1}^{n}q_ka_{k}\Big|^{p}\Big)^{\frac{1}{p}}$.
\end{center}
The sequence space $\Gamma_p$ is of non-absolute type, which means the absolute value of the sequence $|x|=\{|a_n|\}_{n\geq1}$ doesn't belong to $\Gamma_p$ whenever $x\in \Gamma_p$. The following result demonstrates the completeness of the space $\Gamma_p$. Therefore, let us begin with the following proposition.
\begin{Proposition}
The pair $(\Gamma_p, \|\cdot\|_{\Gamma_p})$ is a Banach sequence space of non-absolute type.
\end{Proposition}
\begin{proof}
It is only required to show the completeness of the space $(\Gamma_p, \|\cdot\|_{\Gamma_p})$. Suppose that $\{x^{(i)}\}_{i\geq 1}$, where $x^{(i)}=\{a_n^{(i)}\}_{n\geq 1}$ is a cauchy sequence in $(\Gamma_p, \|\cdot\|_{\Gamma_p})$. Then for any $\varepsilon>0$ there exists a natural number $N_0$ such that $\|x^{(i)}-x^{(j)}\|_{\Gamma_p}<\varepsilon$ for all $i, j>N_0$. Hence by definition, we have
\begin{center}
$\|x^{(i)}-x^{(j)}\|_{\Gamma_p}^{p}=\Big(\displaystyle\sum_{n=1}^{\infty}\gamma_{n} \Big|\displaystyle\sum_{k=1}^{n}q_k(a_{k}^{(i)}-a_{k}^{(j)})\Big|^{p}\Big)<\varepsilon^p$.
\end{center}
Since $\gamma_n>0$, and $q_n>0$ for each $n\in \mathbb{N}$, so one can easily concludes that for each $n\in \mathbb{N}$ the sequence $\{a_n^{(i)}\}_{i\geq 1}$ is Cauchy in $\mathbb{C}$, hence converges. Let $\displaystyle\lim_{j\rightarrow\infty}a_n^{(j)}=a_n$ for a fixed $n\in \mathbb{N}$, where $x=\{a_n\}_{n\geq 1}$. Now fix any $m\in \mathbb{N}$. For a given $\varepsilon>0$, there exists $N_0\in \mathbb{N}$, we have
\begin{center}
$\Big(\displaystyle\sum_{n=1}^{m}\gamma_{n} \Big|\displaystyle\sum_{k=1}^{n}q_k(a_{k}^{(i)}-a_{k})\Big|^{p}\Big)^{\frac{1}{p}}<\varepsilon$.
\end{center}Choosing $m\rightarrow\infty$, we get $\|x^{(i)}-x\|_{\Gamma_p}<\varepsilon$ for all $i>N_0$. This completes the proof. Therefore, $(\Gamma_p, \|\cdot\|_{\Gamma_p})$ is a Banach sequence space of non-absolute type.
\end{proof}
Let us define an infinite matrix $G=(g_{nk})_{n, k\geq 1}$ as below:
 \begin{align*}
g_{nk} &= \left\{
\begin{array}{ll}
    \gamma_n^{\frac{1}{p}}q_k & \quad 1\leq k \leq n,\\ \vspace{0.2cm}
    0 & \quad k>n.
\end{array}\right.
\end{align*}
Note that the matrix $G$ includes more general matrices such as generalized Ces\`{a}ro matrix $C^N$ when one chooses $q_k=1$, and $\gamma_n=\frac{1}{(n+N)^p}$, weighted mean matrix when $\gamma_n=\frac{1}{Q_n^p}$, and etc. These matrices are studied by Chen et al. \cite{CHENLOUR} in connection with Hardy's inequality. Further, it is easy to observe that the sequence space $\Gamma_p$ can be written equivalently as $\Gamma_p=\big\{x=\{a_n\}\in s: ~Gx \in l_p\big\}$. Also since $G$ is lower triangular, so it is invertible, and the inverse matrix $G^{-1}$ is defined as below:
\begin{align*}
g_{nk}^{-1} &= \left\{
\begin{array}{ll}
    \frac{\gamma_k^{-\frac{1}{p}}}{q_n} & \quad k = n, n\geq 1,\\ \vspace{0.2cm}
    -\frac{\gamma_{k}^{-\frac{1}{p}}}{q_n} & \quad k = n-1, n\geq 2,\\
    0 & \quad \mbox{elsewhere}.
\end{array}\right.
\end{align*}
Now we have another proposition given below.
\begin{Proposition}
The space $(\Gamma_p, \|\cdot\|_{\Gamma_p})$ is linearly isomorphic to the $p$-summable sequence space $(l_p, \|\cdot\|_{p})$.
\end{Proposition}
\begin{proof}
We need to prove that there is an existence of a bijective and norm preserving linear map $T:\Gamma_p\rightarrow l_p$. Using the matrix $G$, we define $T$ as $Tx=Gx$ for all $x\in \Gamma_p$. Then $T$ is linear, and for any $y=\{b_n\}_{n\geq 1}\in l_p$, with $b_0=0$ we have $y=Gx$, and hence $x=G^{-1}y$. In fact the terms of the sequence $x=\{a_n\}_{n\geq 1}$ are $a_n=\frac{1}{q_n}(\gamma_n^{-\frac{1}{p}}b_n-\gamma_{n-1}^{-\frac{1}{p}}b_{n-1})$, for $n\in \mathbb{N}$. The map $T$ is also norm preserving as
\begin{align*}
\|x\|_{\Gamma_p}^p&=\displaystyle\sum_{n=1}^{\infty}\gamma_{n} \Big|\displaystyle\sum_{k=1}^{n}q_ka_{k}\Big|^{p}=\displaystyle\sum_{n=1}^{\infty}\gamma_{n} \Big|\displaystyle\sum_{k=1}^{n}\gamma_k^{-\frac{1}{p}}b_k-\gamma_{k-1}^{-\frac{1}{p}}b_{k-1}\Big|^{p}= \sum_{n=1}^{\infty}|b_n|^p=\|y\|_p^p=\|Tx\|_p^p.
\end{align*}One can also easily established that the map $T$ is onto. Therefore, we conclude that there exists a isometric isomorphism between sequence spaces $\Gamma_p$ and $l_p$. This completes the proof.
\end{proof}

\begin{Proposition}
The space $(\Gamma_p, \|\cdot\|_{\Gamma_p})$ is not a Hilbert space except when $p\neq 2$.
\end{Proposition}
\begin{proof}
In case when $p=2$, then it is very easy to verify that the parallelogram identity $\|x+y\|_{\Gamma_p}^2+\|x-y\|_{\Gamma_p}^2=2(\|x\|_{\Gamma_p}^2+\|y\|_{\Gamma_p}^2)$ holds. Hence $\Gamma_2$ is a Hilbert space. Now choose $p\neq 2$, and $x=\{\frac{\gamma_1^{-\frac{1}{p}}}{q_1}, \frac{1}{q_2}(\gamma_2^{-\frac{1}{p}}-\gamma_1^{-\frac{1}{p}}), -\frac{\gamma_2^{-\frac{1}{p}}}{q_3}, 0, 0, \ldots\}$, $y=\{\frac{\gamma_1^{-\frac{1}{p}}}{q_1}, -\frac{1}{q_2}(\gamma_2^{-\frac{1}{p}}+\gamma_1^{-\frac{1}{p}}), \frac{\gamma_2^{-\frac{1}{p}}}{q_3}, 0, 0, \ldots\}$. For such sequences, we have $\|x+y\|_{\Gamma_p}^2=4$, $\|x-y\|_{\Gamma_p}^2=4$, $\|x\|_{\Gamma_p}^2=4^{1/p}=\|y\|_{\Gamma_p}^2$. Therefore, $\|x+y\|_{\Gamma_p}^2+\|x-y\|_{\Gamma_p}^2=8\neq 4^{\frac{p+1}{p}}=2(\|x\|_{\Gamma_p}^2+\|y\|_{\Gamma_p}^2)$ is true for $p\neq 2$. Hence the result follows.
\end{proof}

\begin{Proposition}
The following inclusions results are true:\\
$(a)$ if $\gamma_n=w_n$, $q_n=1$ for all $n\in \mathbb{N}$ then for any $p>1$, $l_p\subset W_p \subseteq X_p$.\\
$(b)$ if $\{Q_n\gamma_n^{\frac{1}{p}}\}_{n\geq 1}\in l_p$ then $l_\infty\subset \Gamma_p$.
\end{Proposition}
\begin{proof}
$(a)$ The result directly follows from the inequality (\ref{IHGP}). Note that the series $\displaystyle\sum_{n=1}^{\infty}w_n$ converges. The inclusion $l_p\subset W_p$ is strict as the sequence $x=\{(-1)^n\}_{n\geq 1}\in W_p$ but $x\notin l_p$.\\
$(b)$ Since $x\in l_\infty$, so for each $n\in \mathbb{N}$, there is a constant $M>0$ such that $\Big|\displaystyle\sum_{k=1}^{n}q_ka_{k}\Big|^p\leq (MQ_n)^p$, where $Q_n=q_1+\ldots+q_n>0$. Hence the desired inclusion follows easily from $\displaystyle\sum_{n=1}^{\infty} \gamma_n\Big|\sum_{k=1}^{n}q_ka_k\Big|^p \leq \sum_{n=1}^{\infty}Q_n^p\gamma_n<\infty$. The inclusion is strict as if we choose the sequences $q=\{\gamma_1^{-\frac{1}{p}}, \frac{\gamma_1^{-\frac{1}{p}}}{2}, 0, \ldots\}$, and $\{\gamma_n\}\in l_1$ then $x=\{(-1)^{n+1}n\}\in \Gamma_p$ but $x\notin l_\infty$.
\end{proof}
In the next result, we obtain the associate space of $\Gamma_p$ for $p>1$. The associate space (known as $\beta$-dual or K\"{o}the-Toeplitz dual in literature) of $X$ (see \cite{MADDOXLN}) is denoted by $X^{'}$ or $X^{\beta}$, and defined as the space of all sequences $y=\{b_n\}_{n\geq 1}$ such that for every $x=\{a_n\}_{n\geq 1}\in X$, the series $\displaystyle\sum_{n=1}^{\infty}a_nb_n$ converges.
\begin{thm}
The associate space of $\Gamma_p$ is the set of all sequences $y=\{b_n\}_{n\geq 1}$ such that $\displaystyle \sup_{n\geq 1}\Big\{\sum_{k=1}^{n-1}|\gamma_k^{-\frac{1}{p}}(b_k-b_{k+1})|^q+ |\gamma_n^{-\frac{1}{p}}b_n|^q \Big\}^{\frac{1}{q}}<\infty$.
\end{thm}
\begin{proof}
Let $x=\{a_n\}_{n\geq 1}\in \Gamma_p$, and $y=\{b_n\}_{n\geq 1}\in (\Gamma_p)'$ be two sequences. Using the idea of Abel transformation, one can write
\begin{align}{\label{DUALI}}
\displaystyle\sum_{k=1}^{n}a_kb_k
& = \sum_{k=1}^{n-1}\Big(\gamma_k^{\frac{1}{p}}\sum_{i=1}^{k}q_ia_i\Big)\Big(\gamma_k^{-\frac{1}{p}}(b_k-b_{k+1})\Big) + \Big(\gamma_n^{\frac{1}{p}}\sum_{i=1}^{n}q_ia_i\Big)\Big(\gamma_n^{-\frac{1}{p}}b_n\Big) = \displaystyle\sum_{k=1}^{n} h_{nk}\sigma_k,
\end{align}where $\sigma_k=\gamma_k^{\frac{1}{p}}\displaystyle\sum_{i=1}^{k}q_ia_i$, and the matrix $H=(h_{nk})_{n, k \geq 1}$ is given as below:
\begin{align*}
h_{nk} &= \left\{
\begin{array}{lll}
    \gamma_k^{-\frac{1}{p}}(b_k-b_{k+1}) & \quad 1\leq k \leq n-1,\\ \vspace{0.2cm}
    \gamma_n^{-\frac{1}{p}}b_n & \quad k=n\\
    0 & \quad k>n.
\end{array}\right.
\end{align*}
Note that for $x\in \Gamma_p$, the sequence $\sigma=(\sigma_k)_{k\geq 1}$ is in $l_p$. Since for each $k\in \mathbb{N}$, the sequence $\{h_{nk}\}_{n=1}^{\infty}$ converges so $H\sigma\in c$, where $c$ is the space of all convergent sequences. Therefore, from (\ref{DUALI}), we conclude that for $x\in \Gamma_p$, sequence $\{a_nb_n\}_{n\geq 1}\in cs$, where $cs$ is the space of all convergent series if and only if $H\sigma\in c$ whenever $\sigma\in l_p$. This yields the fact that $y\in (\Gamma_p)'$ iff $H: l_p\rightarrow c$. It is known that from Stieglitz and Tietz result \cite{STTIZ} that the map $H: l_p\rightarrow c$ is possible if and only if the following conditions hold:\\
$(i)$ $\displaystyle\lim_{k\rightarrow\infty}h_{nk}$ exists for each $k\in \mathbb{N}$, and\\
$(ii)$ $\displaystyle\sup_{n\geq 1} \Big(\sum_{k=1}^{n}|h_{nk}|^q\Big)^{\frac{1}{q}}<\infty$, where $\frac{1}{p}+\frac{1}{q}=1$.\\
Since the first condition is trivially satisfied so the second condition gives the associate space $(\Gamma_p)'$. Hence the theorem is proved.
\end{proof}
For further results on the space $\Gamma_p$, we need to define a sequence $\{u_n\}_{n\geq 1}$ of sequences, where\\
$u_{1}=(\frac{\gamma_{1}^{-\frac{1}{p}}}{q_{1}}, -\frac{\gamma_{1}^{-\frac{1}{p}}}{q_{2}}, 0, 0,\cdots)$, \ldots,
$u_{n}=(0,\cdots 0, \frac{\gamma_{n}^{-\frac{1}{p}}}{q_{n}}, -\frac{\gamma_{n}^{-\frac{1}{p}}}{q_{n+1}}, 0, 0,\cdots)$, \ldots.\\
It is easy to verify that $\|u_n\|_{\Gamma_p}=1$ for each $n\in \mathbb{N}$. Then we have the following result.
\begin{thm}Let $x\in \Gamma_p$ for $p>1$. Then\\
(i) for every $x=\{a_n\}_{n\geq 1}$ we have $x=\displaystyle\sum_{i=1}^{\infty}\Big(\gamma_{i}^{\frac{1}{p}}\displaystyle\sum_{k=1}^{i}q_{k}a_{k}\Big)u_{i}$.\\
(ii) $\Gamma_p$ is separable.
\end{thm}
\begin{proof}
(i) It is sufficient to show that the sequence $\Big\{\displaystyle\sum_{i=1}^{n}\Big(\gamma_{i}^{\frac{1}{p}}\displaystyle\sum_{k=1}^{i}q_{k}a_{k}\Big)u_{i}\Big\}_{n\geq 1}$ norm convergent to $x$ for large $n$, and which can be easily deduce from the following computation.
\begin{align*}
\Big\|x-\displaystyle\sum_{i=1}^{n}\Big(\gamma_{i}^{\frac{1}{p}}\displaystyle\sum_{k=1}^{i}q_{k}a_{k}\Big)u_{i}\Big\|_{\Gamma_p}^p &= \Big\|\{0, \ldots, 0, a_{n+1}+\frac{1}{q_{n+1}}\displaystyle\sum_{k=1}^{n}q_{k}a_{k}, a_{n+2}, \ldots \}\Big\|_{\Gamma_p}^p\\
& = \displaystyle\sum_{k=n+1}^{\infty}\gamma_{k} \Big|\displaystyle\sum_{i=1}^{k}q_ia_{i}\Big|^{p}\rightarrow 0 \mbox{~as~} n\rightarrow\infty.
\end{align*}
(ii) We prove that there exists a countable dense subset of $\Gamma_p$. For each $n=1, 2, \ldots$, and fixed sequences $\{q_n\}$, $\{\gamma_n\}$, let us define a set $$S=\Big\{\displaystyle\sum_{k=1}^{n}\mu_k\Big(q_k\sum_{i=k}^{n}\gamma_{i}^{\frac{1}{p}}u_i\Big)=\Big\{\mu_1, \mu_2, \ldots, \mu_n, -\frac{1}{q_{n+1}}\displaystyle\sum_{k=1}^{n}q_{k}\mu_{k}, 0, 0, \ldots\Big\}: ~~\mu_k\in \mathbb{Q} \Big\}.$$ It is immediate that $S$ is countable and a subset of $\Gamma_p$. Also an element $\mu\in S$ can be written as $\mu=\displaystyle\sum_{i=1}^{n}\Big(\gamma_{i}^{\frac{1}{p}}\displaystyle\sum_{k=1}^{i}q_{k}\mu_{k}\Big)u_{i}$. Let $x\in \Gamma_p$ be any element. Since $\overline{\mathbb{Q}}=\mathbb{R}$, so one can have the following:
\begin{align*}
\big\|x-\mu\big\|_{\Gamma_p}^p &= \Big\|\Big\{a_1-\mu_1, a_2-\mu_2, \ldots, a_n-\mu_n, a_{n+1}+\frac{1}{q_{n+1}}\displaystyle\sum_{k=1}^{n}q_{k}a_{k}, a_{n+2}, \ldots \Big\}\Big\|_{\Gamma_p}^p\\
& = \displaystyle\sum_{k=1}^{n}\gamma_{k}q_k|(a_k-\mu_k)|^p+\displaystyle\sum_{k=n+1}^{\infty}\gamma_{k} \Big|\displaystyle\sum_{i=1}^{k}q_ia_{i}\Big|^{p}\rightarrow 0 \mbox{~as~} n\rightarrow\infty.
\end{align*}
Hence $\Gamma_p$ is separable.
\end{proof}


\NI\textit{Acknowledgement:}
The authors are very much thankful to Prof. David Krej\v{c}i\v{r}\'{i}k (Prague, Czech Republic) for his valuable comments on the first draft of our manuscript.

 \end{document}